\documentclass[a4paper,twoside,11pt]{amsart}

\usepackage{graphicx}
\usepackage[english]{babel}
\usepackage[T1]{fontenc}
\usepackage[ansinew]{inputenc}
\usepackage{geometry}
\usepackage{color} 
 \usepackage{amsmath}
 \numberwithin{equation}{section}

\usepackage{amsmath}
\usepackage{amsthm}
\usepackage{amsfonts}
\usepackage{amssymb}
\usepackage{graphics}
\usepackage{color}
\usepackage{setspace}
\theoremstyle{plain}

\theoremstyle{plain}
\newtheorem{theorem}{Theorem}[section]  
\newtheorem{lemma}[theorem]{Lemma}

\newtheorem{remark}[theorem]{Remark}\newtheorem{example}[theorem]{Example}
\theoremstyle{definition}
\newtheorem{definition}[theorem]{Definition}
   
\newtheorem{problem}[theorem]{Problem}
\theoremstyle{remark}

\newcommand\C{\mathbb C}        
\newcommand\R{\mathbb R}        
\newcommand\N{\mathbb N}         
\newcommand\Ha{\mathbb H}
\newcommand\D{\mathbb D}      
\newcommand\F{\mathcal F}    
\newcommand{\eps}{\varepsilon}

\newcommand{\LandauO}{\mathcal{O}}

\renewcommand\Im{\text{Im}}        
\renewcommand\Re{\text{Re}}        

\newcommand{\supp}{\operatorname{supp}} 
\newcommand{\ID}{{\rm \bf ID}} 
   \usepackage{enumitem} 
	\usepackage[pdfstartview=FitH]{hyperref}

\def\QEQ{{%
    \setbox0\hbox{$\rhd$}%
    \rlap{\hbox to \wd0{\hss$\cdot$\hss}}\box0
}}

\begin{document}
\parindent 0pt 
\definecolor{red}{rgb}{0.9, 0.0, 0.0}
\definecolor{magenta}{rgb}{0, 0.9, 0.0}
\pagestyle{empty} 

\setcounter{page}{1}

\pagestyle{plain}

\title{Problems related to conformal slit-mappings}
\date{\today}
\author{I. Hotta and S. Schlei{\ss}inger}
\thanks{The first author was supported by JSPS KAKENHI Grant no. 17K14205}
\thanks{The second author was supported by the German Research Foundation (DFG), project no. 401281084.}

\keywords{Slit-mapping, univalent function, Loewner chain, monotone convolution, Hilbert transform}

\subjclass[2010]{Primary: 30C35, 46L53, Secondary: 30C80 30C55, 60G51}

\begin{abstract}
In this note we discuss some problems related to conformal slit-mappings. On the one hand, classical Loewner theory leads us to 
questions concerning the embedding of univalent functions into slit-like Loewner chains. On the other hand, a recent result from monotone probability theory motivates the study of univalent functions from a probabilistic perspective.
\end{abstract}

\maketitle

\tableofcontents

\newpage

\section{Introduction}
Let $\D=\{z\in\C\,|\, |z|<1\}$ be the unit disc in the complex plane.
The class $S$ is defined as the set of all univalent (=holomorphic and injective) $f:\D\to\C$ with $f(0)=0$ and $f'(0)=1.$\\
The famous Bieberbach conjecture states that if $f(z)=z+\sum_{n\geq 2}a_nz^n$ belongs to $S$, then $|a_n|\leq n$ for all $n\geq 2.$ 
Bieberbach himself proved the case $n=2$. Later on, Loewner introduced a new method to handle the case $n=3$ 
(\cite{Loewner:1923}).
His approach has been extended and generalized to what is now called \emph{Loewner theory}, and it 
was also used in the final proof of the conjecture by de Branges.

\begin{definition}
A \emph{(normalized radial) Loewner chain} is a family $(f_t)_{t\geq0}$ of univalent functions $f_t:\D\to \C$ with $f_t(0)=0$, 
$f_t'(0)=e^t$, and $f_s(\D)\subseteq f_t(\D)$ whenever $s\leq t.$ We say that a function $f\in S$ can be \emph{embedded} into a Loewner chain 
if there exists a Loewner chain $(f_t)$ with $f_0=f$.
\end{definition}

A Loewner chain is differentiable almost everywhere and satisfies Loewner's partial differential equation:
 \begin{equation}\label{PDE}
  \frac{\partial f_t}{\partial t}(z) = z f '_t(z) p(t,z) \quad 
  \text{ for a.e. $t\geq 0$ and all $z\in\D$}. 
  \end{equation}

 The function $p:[0,\infty)\times \D\to\C$ is a so called Herglotz vector field, i.e., for almost every $t\geq0$,
 $p(t,\cdot)$ maps $\D$ holomorphically into 
 the right half-plane and $0$ onto $1$, and for every $z\in\D$, $t\mapsto p(t,z)$ is measurable. Conversely, 
 every Herglotz vector field 
 uniquely defines a Loewner chain. We refer to \cite[Chapter 6]{Pom:1975} for these statements. 
 Pommerenke proved the following nice result.

 \begin{theorem}[Theorem 6.1 in \cite{Pom:1975}]\label{thm:e}
 Every $f\in S$ can be embedded into a Loewner chain.
 \end{theorem}
 \begin{remark}
  Loewner chains can also be regarded in $\C^n$ or on complex manifolds.  We refer to  \cite{Fia17} for embedding problems 
  of biholomorphic mappings on the Euclidean unit ball in $\C^n$ and to the recent result from \cite{FW18}, 
  which shows that the analogue of Pommerenke's theorem fails in higher dimensions.
 \end{remark}

A \emph{slit in $\C$} is a Jordan curve $\Gamma$ connecting some
 $z_0\in \C$ to $\infty.$
We call $f\in S$ a \emph{slit mapping} if $f(\D)$ is the complement of a slit.  Loewner's original result focuses on slit mappings.
  
 \begin{theorem}[\cite{Loewner:1923}]\label{loe}
 Let $f\in S$ be a slit mapping. Then $f$ can be embedded into exactly one Loewner chain $\{f_t\}_{t\geq0}$. 
 There exists a continuous $\kappa:[0,\infty)\to\partial\D$ such that
 \begin{equation}\label{slit_LE} \frac{\partial f_t}{\partial t}(z) = z f '_t(z)
 \frac{\kappa(t)-z}{\kappa(t)+z} \quad \text{ for every $t\geq 0.$} 
 \end{equation}
 
 \end{theorem}
 \begin{remark} Each $f_t$ maps $\D$ onto the complement of a subslit. Denote by $\gamma(t)$ the tip of this slit.
 Then, for each $t\geq0$, $f_t^{-1}$ can be extended continuously to $\gamma(t)$ and the driving function $\kappa$ 
 can be written as \begin{equation}\label{kappa}
           \kappa(t) = f_t^{-1}(\gamma(t)).
          \end{equation}
 \end{remark}
 
In this paper, we address some embedding problems in Section \ref{Sec_emb}, 
which are all motivated by Theorem \ref{loe}. 
In Section \ref{Sec_mon}, we look at slit mappings from a probabilistic point of view.
 
\newpage
 
 \section{Embedding problems}\label{Sec_emb}
 
   Note that in Theorem \ref{loe}, the Loewner chain $(f_t)$ is \emph{uniquely determined}
   and \emph{differentiable everywhere}
  (right-differentiable at $t=0$). This leads us to a couple of subclasses of $S$ related 
  to embedding problems. Before defining these classes, we point out how to recover the first element 
  of a Loewner chain from the Loewner equation.\\
  
   Loewner's ordinary differential equation is the following analogue to \eqref{PDE}:
    \begin{equation}\label{ODE}
  \frac{\partial \varphi_{s,t}}{\partial t}(z) = -\varphi_{s,t}(z)\cdot p(t,\varphi_{s,t}(z)) \quad 
  \text{ for a.e. $t\geq s$ with $\varphi_{s,s}(z)=z$ } 
  \end{equation}
  for all $z\in \D$. The solution $(\varphi_{s,t})_{0\leq s \leq t}$ is a family of univalent functions 
  $\varphi_{s,t}:\D\to\D$.\\
  
  If $(f_t)$ satisfies \eqref{PDE}, then $\varphi_{s,t}$ is given by 
  $\varphi_{s,t}=f_t^{-1}\circ f_s$ and the functions $(\varphi_{s,t})_{0\leq s \leq t}$ 
  are thus called the transition mappings of the Loewner chain.\\
  Conversely, if $\varphi_{s,t}$ is the solution to \eqref{ODE}, then, for every $s\geq0$, 
  \begin{equation}\label{Loewner_ODE}
  f_{s} = \lim_{t\to\infty}e^t\varphi_{s,t}
  \end{equation}
  locally uniformly on $\D$; see \cite[Theorem 6.3]{Pom:1975}.  Thus, the first element of a Loewner chain can also be regarded as the infinite time limit 
  of the solution of \eqref{ODE} for $s=0$.

  \begin{remark}\label{ODE_reach}
   If $D\subset \D$ is a simply connected domain with $0\in D$, then there exists $T>0$ 
   and a Herglotz vector field $p(t,z)$ such that the solution $\varphi_{0,t}$ of \eqref{ODE} 
   satisfies $\varphi_{0,T}(\D)=D$. This follows basically from Theorem \ref{thm:e} and is mentioned 
   as an exercise in \cite[Section 6.1, Problem 3]{Pom:1975}.\\ 
   The above statement is equivalent to the following: Let $f\in S$ such that $f(\D)$ is bounded. 
   Then there exists $T>0$ and a Loewner chain $(f_t)$ such that $f_0=f$ and $f_T(\D)=e^T\D$.
  \end{remark}


 \subsection{Differentiability}
Let us call a Loewner chain $(f_t)$ \emph{differentiable} if $t\mapsto f_t(z)$ is differentiable at every $t\geq 0$ 
for every $z\in\D$. We define the class
 	$$S_d := \{f\in S\,|\, \text{$f$ can be embedded into a differentiable Loewner chain}\}.$$ 
 	Every slit mapping belongs to $S_d$ due to Theorem \ref{loe}. 
 	Another simple example can be obtained as follows. Assume that $f(\D)$ is bounded by a closed Jordan curve. 
 	Then we can first connect this curve to $\infty$ by a Jordan arc, and now  
 	erase the two curves to obtain a Loewner chain satisfying \eqref{slit_LE}.
	
	 \begin{figure}[h]\label{jordan_figure}
 \begin{center}
 \includegraphics[width=0.35\textwidth]{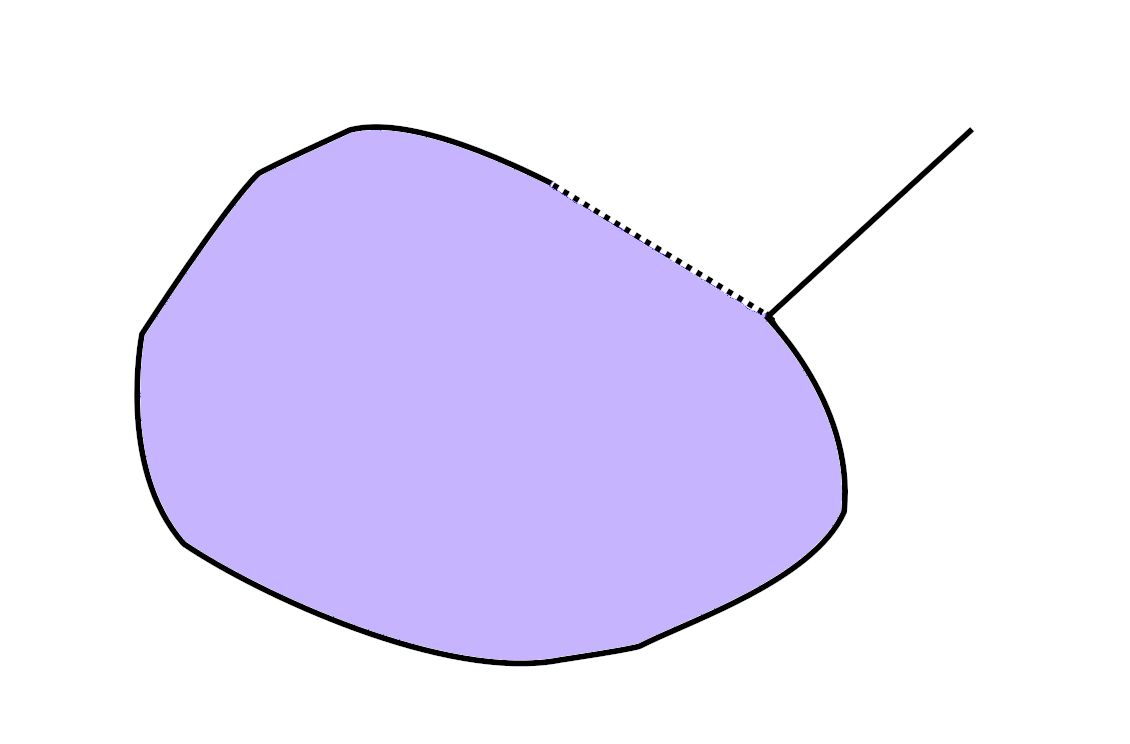}
 \caption{A Jordan domain (blue) connected to $\infty$ by a Jordan arc.}
 \end{center}
 \end{figure}
	
	Suppose that $f\in S$ maps $\D$ onto the complement of two disjoint slits. Then we can embed $f$ into a
 	Loewner chain by erasing a piece of the first slit in some time interval $[0,T_1]$, then a piece of the second slit in 
 	an interval $[T_1,T_2]$, etc. In this case, $(f_t)$ is not differentiable at $t=T_1$. 
 	However, one can also erase the slits simultaneously and then the corresponding Loewner
 	chain is differentiable everywhere. This is true for any $f$ mapping $\D$ onto 
 	the complement of finitely many slits. These statements follow from \cite[Theorem 2.31]{Bo}. \\
 	
 However, not every $f\in S$ belongs to $S_d$.
 
\newpage
 
 \begin{theorem}\label{S_d} There exists $f\in S\setminus S_d$.
 \end{theorem}
 \begin{proof}
 
   Let $D\subsetneq\C$ be a simply connected domain with $0\in D$ and let $f:\D\to D$ be the conformal mapping
 with $f(0)=0$ and $f'(0)>0$. In what follows, the number $c=f'(0)$ will be called the \emph{capacity} of $D$ 
 and $f$ its \emph{normalized conformal mapping}.\\
 Koebe's one-quarter theorem implies that $D$ contains a disc centered at $0$ with radius $c/4$, 
 see \cite[Theorem 2.3]{duren83}.\\
 
 Consider the topological sine $J=\{x+i\sin(1/x)\,|\, x>0\}\cup \{ix\,|\, x\in[-1,1]\}.$
 We connect $J$ by a Jordan curve $\beta_1$ starting at $i$ and staying in $\C \setminus J$ otherwise.
 We do the same for a second curve $\beta_2$ starting at $-i;$ see the figure below.

 \begin{figure}[h]\label{topsin}
 \begin{center}
 \includegraphics[width=0.60\textwidth]{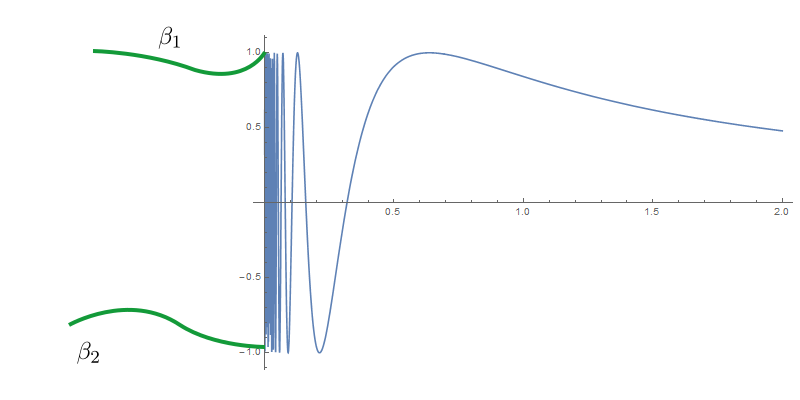}
 \caption{The sets $J, \beta_1, \beta_2$.}
 \end{center}
 
 \end{figure}
 
 Now we translate the set $J\cup \beta_1 \cup \beta_2$ such that $0$ belongs to the complement, and then 
 scale it (we keep the notation for these new sets) such that $\C\setminus (J\cup \beta_1)$ has capacity $1$. 
 Denote by $h_1$ the normalized conformal mapping of $\C\setminus (J\cup \beta_1)$.\\
 
 Next we look at the domain $\C\setminus (J\cup \beta_2)$. If we change it by 
 extending or shortening the curve $\beta_2$, then the capacity changes continuously due to Carath\'{e}odory's 
 kernel theorem. We can extend $\beta_2$ to a neighbourhood of $0$
 to make the  capacity as small as we like, due to Koebe's one-quarter theorem. 
 Furthermore, the domain $\C\setminus J$ has a capacity larger than $1$. \\ 
 Hence, the intermediate value theorem implies that we can extend or shorten $\beta_2$ 
 (we keep the same notation) such that $\C\setminus (J\cup \beta_2)$ has capacity $1$. 
 Let $h_2$ be the normalized conformal mapping of $\C\setminus (J\cup \beta_2)$. \\

 Then $h_1,h_2\in S$ and we can use Theorem \ref{thm:e} to obtain a Loewner chain $\{f_{1,t}\}_{t\geq0}$ 
 with $f_{1,0}=h_1$ and  a Loewner chain $\{f_{2,t}\}_{t\geq 0}$ with $f_{2,0}=h_2.$ 
 It is easy to see that $f_{1,t}$ is unique, as $f_{1,t}$ must erase the curve $\beta_1$ for $t\in [0,T]$ for some $T>0.$ 
 The function $f_{1,T}$ maps $\D$ onto $\C\setminus J.$
  For $t> T,$ the Loewner chain erases the topological sine. Similarly, $f_{2,t}$ is unique and $f_{1,t}=f_{2,t}$ for all $t\geq T.$ \\

 We show that there exist continuous 
 $\kappa_j:[0,T)\cup (T,\infty)\to\partial\D$ such that $f_{j,t}$ satisfies
 \eqref{slit_LE} for every $t\in [0,T)\cup (T,\infty).$ This is clear for $t>T$, as 
 $f_{j,t}$ is simply a slit mapping then and we can apply Theorem \ref{loe}.\\
 Next it follows from \cite[Proposition 2.14]{MR1217706} that 
 $f_{1,T}^{-1}(\beta_j)=f_{2,T}^{-1}(\beta_j)$ is a curve in $\D$ with one endpoint $K_j$ in $\partial\D$.  
 Moreover, \cite[Proposition 2.14]{MR1217706} also states that $K_1\not=K_2$.\\
 Now we conclude that there exist continuous $\kappa_j:[0,T]\to\partial\D$ such that 
 $f_{j,t}$ satisfies \eqref{slit_LE} on $[0,T]$ (with a left-derivative for $t=T$). 
 Furthermore, $\kappa_j(T)=K_j$. This follows readily from the proof of Loewner's theorem. 
 Alternatively, we can regard the family $(f_{j,T}^{-1}\circ f_{j,T-t})_{t\in[0,T]}$. It describes the growth 
 of the slit $f_{1,T}^{-1}(\beta_j)$ and satisfies the time-reversed version of Loewner's differential equation 
 with the Herglotz vector field as in \eqref{slit_LE} with continuous driving function, 
 see \cite[Theorem 2.22]{Bo}. It follows that $f_{j,t}$ satisfies \eqref{slit_LE} with continuous 
 $\kappa_j:[0,T]\to\partial\D$ on $[0,T]$.\\

 Now we show that either $f_{1,t}$ or $f_{2,t}$ is not differentiable at $t=T$. 
 Assume the opposite and fix some $z\in\D\setminus\{0\}$. 
 Then $t\mapsto h_t(z):=f_{1,t}(z)-f_{2,t}(z)$ is differentiable for all $t>0.$ We have $\frac{d}{dt}h_t(z)=0$ for all $t>T.$\\
 Furthermore, we see that the limit
 $$
 \lim_{t\uparrow T}\frac{d}{dt}h_t(z) = 2z^{2}f_{T}'(z)\frac{K_{1}-K_{2}}{(K_{1} + z)(K_{2}+z)}
 $$ 
 exists and is different from $0.$ Hence, $\frac{d}{dt}h_t(z)$ 
 has a first kind discontinuity at $t=T.$ This is a contradiction to Darboux's theorem (applied to the real or imaginary part of 
 $\frac{d}{dt}h_t(z)$).
 \end{proof}

 We remark that there is a wide range of examples of $f \in S_{d}$ 
 whose boundary $\partial f(\D)$ is not locally connected, 
 i.e., $f$ does not have a continuous extension to $\overline{\D}$.
 In fact, typical known subclasses of $S$ in the theory of univalent functions 
 (e.g. close-to-convex functions) are contained in $S_{d}$ (see e.g. Section 3.4 in \cite{Hotta:oldandnew}).
  
 \subsection{Unique embeddings}
 Next we define 
 $$S_u := \{f\in S\,|\, \text{$f$ can be embedded into exactly one Loewner chain}\}.$$ 
  Note that all slit mappings belong to $S_u$. Clearly, there is only one way 
  how to remove a slit by a Loewner chain. 
  The proof of Theorem \ref{S_d} 
  implies that there exists $f\in S_u$ which is not a slit mapping. Roughly speaking, 
  the complement $\C\setminus f(\D)$ must be ``thin'' for $f\in S_u$. One might think that 
  $\C\setminus f(\D)=\partial f(\D)$ for such mappings. However, this is not true due to the next example.
  
  \begin{example}
Let $f\in S$ such that $\partial f(\D)$ is an infinite spiral  $\gamma:(0,1)\to \C$ surrounding a disc $D$, i.e. 
$\gamma(t)\to\infty$ as $t\downarrow 0$ and the set of all accumulation points 
$\lim_{n\to\infty}\gamma(t_n)$ with $t_n\uparrow1$ is equal to the circle $\partial D$. Then $f\in S_u$ and 
$\partial f(\D)\subsetneq \C\setminus f(\D)$, as the interior of $D$ does not belong to $\partial f(\D)$.
  \end{example}
 \begin{figure}[h]
 \begin{center}
 \includegraphics[width=0.5\textwidth]{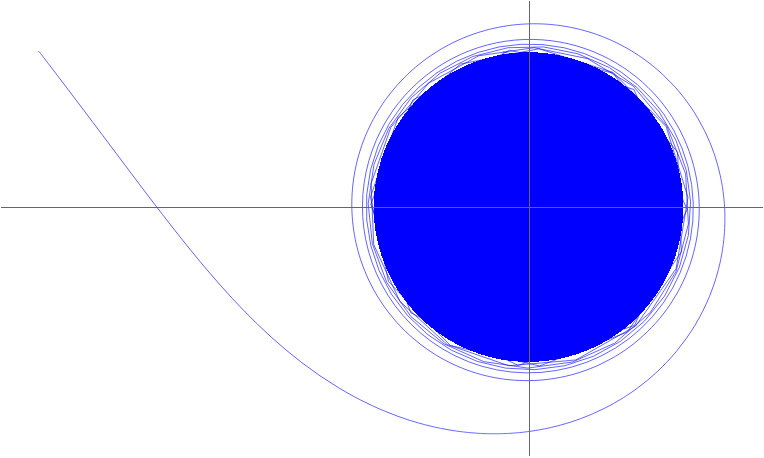}
 \caption{Infinite spiral enclosing a disc.}
 \end{center}
 \end{figure}
  
  The following lemma is quite useful for constructing Loewner chains.
  
    \begin{lemma}\label{ODE_reach2}
   Let $f\in S$ and $D=f(\D)$. Assume that $E\subsetneq \C$ is a simply connected domain with 
   $D\subsetneq E$. Then there exists a Loewner chain $(f_t)$ and $T>0$ such that 
   $f_0=D$ and $f_T=E$.
   \end{lemma}
   \begin{proof}
   Let $g:\D\to E$ be a conformal mapping with 
   $g(0)=0$ and $g'(0)=e^{T}$ for some $T>0$.  There exists a Loewner chain 
   $(h_t)_t$ such that $h_0=e^{-T}g$ due to Theorem \ref{thm:e}. Let $p_1(t,z)$ be the corresponding Herglotz vector field.\\ 
   Write $f=g\circ \varphi$. By Remark \ref{ODE_reach}, there exists a Herglotz vector field $p_2(t,z)$ 
   such that the solution $\varphi_{0,t}$ of \eqref{ODE}  satisfies $\varphi_{0,T}=\varphi$. 
   Now consider the Herglotz vector field $p(t,z)$ defined by 
   \[p(t,z)=p_2(t,z)\quad \text{for $t\leq T$ \quad and} \quad 
   p(t,z)=p_1(t-T,z)\quad \text{for $t>T$ and all $z\in\D$}.\] Let $(f_t)$ be the corresponding Loewner chain 
   with transition 
   mappings $(\psi_{s,t})$. Then $\psi_{0,t}=\varphi_{0,t}$ for all $t\in[0,T]$. 
   We have $f_0=f_T\circ \varphi_{0,T}=f_T\circ \varphi$ 
   and $f_T=\lim_{t\to\infty}e^t \psi_{T,t}=e^Th_0=g$. 
   Hence, $f_T(\D)=E$ and $f_0(\D)=(f_T\circ \varphi)(\D)=f(\D)=D$. 
   \end{proof}

 \begin{theorem}\label{thm_u}  Let $f\in S_u$ and let $D=f(\D)$. Then $D$ has the following properties:
 \begin{itemize}
 \item[(a)] $D$ is unbounded.
 \item[(b)] $\partial D$ is connected.
  \item[(c)] Let 
  \[\mathcal{C}(D)=\{C\subset \partial D\,|\, \text{$C$ is connected, unbounded, and closed}\}.\] 
  If  $C_1, C_2\in \mathcal{C}(D)$, then $C_1\subseteq C_2$ or
  $C_2\subseteq C_1$.
 \end{itemize}
 \end{theorem} 
 \begin{proof}${}$
 \begin{itemize}
 \item[(a)]  Assume that $D$ is bounded. Then we can embed $f$ into a Loewner chain $(f_t)$ such that 
  $f_T(\D)=e^T\D$ for some $T>0$, see Remark \ref{ODE_reach}. As $\D$ 
  can be embedded into many Loewner chains, we conclude $f\not\in S_u$, a contradiction.
\item[(c)] Due to (a), the set $\mathcal{C}(D)$ is non-empty. Let $C_1, C_2\in \mathcal{C}(D)$.
  Let $(f_t)$ be the unique Loewner chain with $f_0=f$ and let $D_t=f_t(\D)$.\\
  Let $E_1$ and $E_2$ be the connected component
  of $\C\setminus C_1$ and $\C\setminus C_2$ respectively containing $D$.\\
  Then $E_1$ and $E_2$ are simply connected domains and due to Lemma \ref{ODE_reach2},
  there exist $T_1, T_2$ such that $E_1=D_{T_1}$, $E_2=D_{T_2}$. This implies
  $E_1\subseteq E_2$ or $E_2\subseteq E_1$, say we have $E_1\subseteq E_2$. We need to show that  $C_2\subseteq C_1$.
Assume that this is not true. Then there exists a point $p\in C_2$ and $p\not\in C_1$.\\
Now note that $\partial D\setminus C_1 \subseteq E_1$. Hence $p\in E_1$ and thus $p\in E_2$. But $p$ also belongs to
$C_2$ and thus to the complement of $E_2$, a contradiction. 

 \item[(b)] Assume that $\partial D$ has at least two connected components $C_1, C_2$. Then 
 both components are unbounded, otherwise $D$ would not be simply connected. Hence,
 $C_1,C_2\in \mathcal{C}(D)$ with $C_1\cap C_2 = \emptyset$, a contradiction to (c).
 \end{itemize}
  \end{proof}

 In case $f\in S_u$ maps $\D$ onto the complement of a slit $\gamma$, the elements of 
 $\mathcal{C(D)}$ are simply subslits of $\gamma$. We see that in the general case, each $p\in \partial D$ 
 is connected to $\infty$ within $\partial D\cup\{\infty\}$ 
 in a unique way, i.e. there is a smallest connected closed subset of $\partial D\cup\{\infty\}$
 containing $p$ and $\infty$.

 \subsection{Slit equation}
 Finally, we can look at the special form of Loewner's differential equation appearing in Theorem \ref{loe}.
 	$$S^1_s :=
 	\{f\in S\,|\, \text{$f$ can be embedded into \eqref{slit_LE} for continuous $\kappa$}\}
 	\subsetneq S_d.$$ 
 If $f\in S$ is a  two-slit mapping, then $f\in S_d$ but $f\not\in S^1_s$, which shows that 
 $S^1_s$ is a proper subset of $S_d$.\\

 The class $S^1_s$ (and its variations) has been studied intensively in the literature.  
 
 \begin{itemize}
 \item[\textbullet] Pommerenke characterizes Loewner chains corresponding to $S^1_s$ via the 
 ``local growth property'', see  \cite[Theorem 1]{MR0206245}.
 \item[\textbullet] Every slit mapping belongs to $S^1_s$. 
 However, continuous driving functions can also 
 create non-slit mappings. For example, every $f\in S$ such that $f(\D)$ is a Jordan domain 
 belongs to $S^1_s$ due to the Loewner chain depicted in Figure \ref{jordan_figure}. 
 One can even generate spacefilling curves by continuous $\kappa$, see \cite{LR12}.\\
  The set of all continuous driving functions that correspond to slits in this way is 
not known explicitly.  However, there are several partial results into that direction. 
 Roughly speaking, if $\kappa$ is smooth enough, e.g. continuously differentiable, then 
 $f$ is a slit mapping. We refer to the recent work \cite{ZZ18} and the references therein for such results.
 \item[\textbullet] Loewner's slit equation can be seen as a machinery 
transferring a simple curve $\Gamma$ into a continuous function 
$\kappa:[0,\infty)\to \partial\D$.  
 This process  \[\Gamma \longrightarrow \kappa\]
encodes ``difficult'' two-dimensional objects into one-dimensional ones. 
It seems that this relationship is both rather mysterious and (therefore) quite powerful.
In case of the celebrated 
Schramm-Loewner evolution, certain planar random curves, whose distributions are not easy 
to understand, are simply transferred into $\kappa(t)=e^{i\sqrt{\kappa}B_t}$, 
where $\kappa\geq0$ is a parameter 
and $B_t$ is a standard Brownian motion. For an introduction to SLE, we refer to \cite{lawler05}.

 \end{itemize}

 We obtain a second class by requiring 
 that  \eqref{slit_LE} should hold only almost everywhere.

 	$$S^2_s := \{f\in S\,|\, \text{$f$ can be embedded into \eqref{slit_LE} for 
 	measurable $\kappa$}\}.$$

 Recall that a domain $D\subset \C$ is simply connected if and only if $\hat{\C}\setminus D$ is connected. 
If, in addition, $\hat{\C}\setminus D$ is pathwise connected, then one can erase slits in the complement 
of $D$. Pommerenke constructed a Loewner chain in this way to obtain the following result.
 
 \begin{theorem}[Theorem 2 in \cite{MR0206245}]\label{pom_p}
 Let $f\in S$ such that $\hat{\C}\setminus f(\D)$ is pathwise connected. Then $f\in S^2_s$.
 \end{theorem} 
 
\subsection{Problems} 

Thinking of the idea of the proof of Theorem \ref{pom_p}, we are led to the following question.

 \begin{problem}
   Let $f\in S$ such that $\hat{\C}\setminus f(\D)$ is pathwise connected. Is it possible to embed 
  $f$ into a differentiable Loewner chain by simultaneously erasing slits in the complement?
 \end{problem}

 	 \begin{problem}
  Pommerenke asks in \cite{MR0206245}: Is $S=S^2_s$? 
 \end{problem}

This question is interesting from a control theoretic point of view.

Denote by $\mathcal{P}$ the Carath\'{e}odory class of all holomorphic functions
$p:\D\to \C$ with $\Re(p(z))>0$ for all  $z\in\D$ and $p(0)=1$.
The class $\mathcal{P}$ can be characterized by the Riesz-Herglotz representation formula:
\begin{equation*}   \mathcal{P} = \left\{
\int_{\partial \D} \frac{u + z}{u -z} \, \mu(du) \,|\, \mu \; 
\text{is a probability measure on $\partial \D$} \right\}.  \end{equation*}
The extreme points of the class $\mathcal{P}$ are thus given by all functions of the form
$\frac{u+z}{u-z}$ for some $u\in \partial \D$. 
Hence, in view of \eqref{Loewner_ODE}, 
a result like $S=S^2_s$ could be interpreted as a  ``bang-bang principle'' for the Loewner 
equation.

 \begin{problem}
 Let $f\in S_u$ be embedded into its unique Loewner chain $(f_t)$. 
 How does the Loewner equation for $(f_t)$ look like?
 \end{problem}

  Note that an example of $f\in S_u$ whose Loewner equation does not have the form 
  \eqref{slit_LE} for measurable $\kappa$ would prove $S\not=S^2_s$. 
	
 \begin{problem}
Let $f\in S$ such that $D=f(\D)$ satisfies $(a)$-$(c)$ from Theorem \ref{thm_u}. Is it true that $f\in S_u$?
\end{problem}

 \begin{problem}
Is it true that the set $S_u\cap S_d$ contains only slit-mappings?
 \end{problem}
 
     \begin{problem}
Let $f\in S$ but $f\not\in S_u$. Is it true that $f$ can be embedded into infinitely 
(uncountably) many Loewner chains?
 \end{problem}
 
\newpage 

\section{Measures related to univalent slit mappings}\label{Sec_mon}

Holomorphic functions $f:\D\to\C$ also arise in probability theory. 
Such mappings encode probability measures $\mu$ on the unit circle $\partial \D$ or on $\R$. 
The univalence of such functions has a certain meaning in non-commutative probability theory, 
which will be explained in Section \ref{qu}.\\

This correspondence motivates two questions:\\
How can the property that $f(\D)$ has the form $\D\setminus \gamma$, where $\gamma$ is a simple curve, 
be translated into properties of the measure $\mu$?\\
How are the questions from Section \ref{Sec_emb} translated if we pass from non-commutative to classical probability theory?\\

Instead of the unit disc and the normalization $f(0)=0$, we prefer to use the upper half-plane $\Ha=\{z\in\C\,|\, \Im(z)>0\}$ and a normalization 
at the boundary point $\infty$. Then the probability measures will be supported on $\partial\Ha=\R$. \\

We give a partial answer to the first question in  Section \ref{p}. 
In Section \ref{qu} we address the second question
and explain the deeper connection of univalent mappings to non-commutative probability theory.

\subsection{Univalent Cauchy transforms}\label{p}

Let $\gamma:[0,1]\to \overline{\Ha}$ be a simple curve with $\gamma(0)\in \R$ and 
$\gamma(0,1]\subset\Ha$. Then there exists a unique conformal mapping 
$f:\Ha\to\Ha\setminus \gamma(0,1]$ having the 
 \emph{hydrodynamic normalization} 
\[f(z) = z - \frac{c}{z} + \LandauO(|z|^{-2})
\]
for some $c> 0$ as $z\to\infty$. The value $c$ is also called the \emph{half-plane capacity} of the slit.\\

The Cauchy transform (or Stieltjes transform) of a probability measure $\mu$ on $\R$ is given by 
$$G_{\mu}(z)=\int_\R\frac{1}{z-t}\,\mu(dt), \qquad z\in\C\setminus \R.$$
We define the 
$F$--transform of $\mu$ simply as $F_\mu:\Ha\to\Ha, F_\mu(z) := 1/G_\mu(z)$. $F$-transforms can be characterized in the following way.

\begin{theorem}[Proposition 2.1 and 2.2 in \cite{M92}] Let 
$F\colon \Ha \to \overline{\Ha}$ be holomorphic. Then the followings are equivalent.
\begin{itemize}
\item[(a)] There exists a probability measure $\mu$ on $\R$ such that $F=F_\mu$.
\item[(b)] $\lim_{y\to \infty}\frac{F(i y)}{i y}=1$.
\item[(c)] $F$ has the Pick-Nevanlinna representation
\begin{equation*}
F(z) =  z + b + \int_{\R}\frac{1+tz}{t-z}\rho({\rm d}t),
\end{equation*}
where $b \in \R$ and $\rho$ is a finite, non-negative Borel measure on $\R$.
\end{itemize}
\end{theorem}

We conclude that every univalent slit mapping $f:\Ha\to\Ha\setminus \gamma(0,1]$ with hydrodynamic normalization 
is the $F$-transform of a probability measure $\mu$, i.e. $f=F_\mu$. We are thus led to the problem of characterizing those $\mu$ whose 
$F$-transforms are univalent slit mappings.\\

Consider again an arbitrary probability measure $\mu$ on $\R$.
Due to Fatou's theorem, the following radial limits exist almost everywhere on $\R$:
\begin{equation*}\hat{\mathcal{H}}_\mu(x) := \lim_{\eps\downarrow 0} \hat{\mathcal{H}}_{\eps,\mu}(x), \quad
\hat{\mathcal{H}}_{\eps,\mu}(x):=\frac {1}{\pi }\Re\; G_\mu(x+i\eps).
\end{equation*}

The Hilbert transform of $\mu$ is defined by 
\begin{equation*}\mathcal{H}_\mu(x) := 
\lim_{\eps\downarrow 0}\mathcal{H}_{\eps,\mu}(x), \quad 
\mathcal{H}_{\eps,\mu}(x) := \frac {1}{\pi }\int_{|x-t|> \eps}\frac {1}{x-t}\,\mu(dt).
\end{equation*}
$\mathcal{H}_\mu$ is also defined for almost every $x\in\R$.
The Sokhotski-Plemelj formula implies that $\mathcal{H}_\mu(x)$ 
and $\hat{\mathcal{H}}_\mu(x)$ coincide. As this equality is usually stated to hold almost everywhere on $\R$ (see \cite[Theorem F.3]{MR2953553} or \cite[Sections 2.5, 3.8]{CMR06}),
 we include the short proof of the pointwise equality needed in our situation. 

\begin{lemma}\label{equal_Hilbert}Let $\mu$ be an absolutely continuous probability measure with compact
support and continuous density $f(x)dx$. 
Let $x\in\R$. Then $\mathcal{H}_\mu(x)$ exists if and only if
 $\hat{\mathcal{H}}_\mu(x)$ exists. \\
If these limits exist, then $\mathcal{H}_\mu(x)=\hat{\mathcal{H}}_\mu(x)$.
\end{lemma}
\begin{proof}
First, we consider the relevant integrals and change $t$ to 
$t=x+\eps u$, which gives
\begin{eqnarray*}&& \int_{\R}\frac1{x-t+i\eps} f(t)\,dt - 
\int_{|x-t|>\eps} \frac1{x-t} f(t)\,dt \\
&=& 
\int_\R \frac{f(x+\eps u)\eps du}{-\eps u+i\eps} - 
\int_{|\eps u|>\eps}\frac{f(x+\eps u)\eps du}{-\eps u}=
\int_\R \left(\frac{1}{i-u} + 
\chi_{|u|>1}\frac{1}{u}\right) \,f(x+\eps u)du.
\end{eqnarray*}

Denote the function in  parentheses by $g(u)$. 
For $|u|\leq 1$, we have $|g(u)|=\frac{1}{|i-u|}\leq 1$, 
and if $|u|>1$, then 
$|g(u)|=\frac{1}{|i-u||u|}\leq \frac1{|u|^2}.$ So $g$ is integrable and a direct calculation 
yields
 $\int_\R g(u)du=-i\pi$. 
Let $[a,b]$ be a compact interval containing the support of $\mu$. Then 
$|f(x+\eps u) - f(x)|$ is uniformly bounded by $2\max_{t\in[a,b]}|f(t)|$ and the dominated convergence theorem implies that 
\[\int_\R (f(x+\eps u) - f(x)) g(u)\,du \to 0\]
as $\eps \to 0$. Hence, 
$\Re \int_\R  f(x+\eps u)  g(u)\,du \to 0$ as $\eps \to 0$.
So,  $\lim_{\eps\downarrow 0} \hat{\mathcal{H}}_{\eps,\mu}(x)$ exists if and only if $\lim_{\eps\downarrow 0} \mathcal{H}_{\eps,\mu}(x)$ exists, and if these limits exist, then they coincide.
\end{proof}

We first look at the case where the slit does not start at $0$.

\begin{theorem}\label{slit_measures} Let $\mu$ be a probability measure on $\R$ such that $F_\mu$ is univalent. \\
Then $F_\mu$ maps $\Ha$ conformally onto $\Ha\setminus \gamma$, 
where $\gamma$ is a slit starting at $C\in\R\setminus\{0\}$, if and only if the following conditions 
are satisfied:
\begin{itemize}
 \item[(a)] $\supp \mu= \{x_0\}\cup[a,b]$, where $\mu$ 
 has a continuous density $d(x)$ on the compact interval $[a,b]$ and an atom at some $x_0\in \R\setminus [a,b].$  
 Furthermore, $d(a)=d(b)=0$ and $d(x)>0$ in $(a,b)$.
 \item[(b)] $\mathcal{H}_\mu$ is defined and continuous on $\R\setminus\{x_0\}$ with $\mathcal{H}_\mu(a)=\mathcal{H}_\mu(b)=\frac1{\pi C}$.
 \item[(c)] There exists a decreasing homeomorphism $h:[a,b]\to[a,b]$ with 
 \[d(h(x))=d(x) \qquad \text{and} \qquad \mathcal{H}_\mu(h(x))=\mathcal{H}_\mu(x)\]
 for all $x\in [a,b]$. 
 \end{itemize}
\end{theorem}
\begin{proof}
``$\Longrightarrow$'':\\
As the domain $\Ha\setminus \gamma$ has a locally connected boundary, the mapping $F_\mu$ can be extended continuously
to $\overline{\Ha};$ see \cite[Theorem 2.1]{MR1217706}.\\
There exists an interval $[a,b]$ such that $F_\mu([a,b])=\gamma$ and there is a unique $u\in (a,b)$ such that
$F_\mu(u)$ is the tip of the slit. All points $[a,u]$ correspond to the left side, all points $[u,b]$ to the right side 
of $\gamma.$ (This orientation follows from the behaviour of $F_\mu(x)$ as  $x\to\pm\infty.$) 
Hence, there
exists a unique homeomorphism $h:[a,b]\to[a,b]$ with $h(u)=u,$ $h[a,u]=[u,b]$ such that 
$F_\mu(h(x))=F_\mu(x)$ for 
all $x\in[a,b].$\\
Furthermore, $F_\mu$ has exactly one zero $x_0\in\R\setminus[a,b]$ on $\R$, as the slit does not start at $0$. As $C=F_\mu(a)=F_\mu(b)$, we have $x_0 < a$ if and only if $C>0$.\\ 

 It follows from the Stieltjes-Perron inversion formula, see \cite[Theorems F.2, F.6]{MR2953553}, that $\supp \mu=\{x_0\}\cup [a,b]$ and that $\mu$ is absolutely continuous on
 $[a,b]$ and its density 
 $d(x)$ satisfies
 \[d(x) =  \lim_{\eps\to 0} -\frac{1}{\pi}\Im(1/F_\mu(x+i\eps))=-\frac{1}{\pi}\Im(1/F_\mu(x)).\] Hence,
 $d(h(x)) = d(x)$ for all $x\in[a,b]$, $d(x)>0$ on $(a,b)$, and $d(a)=d(b)=0$.\\
Let $\lambda=\mu(\{x_0\})$. Then we have
\[ \frac1{\pi} \Re\left(1/F_\mu(x)\right)
=\hat{\mathcal{H}}_\mu(x)=\hat{\mathcal{H}}_d(x)+\frac{\lambda}{\pi(x-x_0)} = \mathcal{H}_d(x)+\frac{\lambda}{\pi(x-x_0)}=\mathcal{H}_\mu(x)\]
for every $x\in\R\setminus\{x_0\}$ due to Lemma \ref{equal_Hilbert}. Here, $\mathcal{H}_d$ and $\hat{\mathcal{H}}_{d}$ are defined by replacing $\mu(dt)$ by $d(t)dt$ in the integration, and formally, we apply Lemma \ref{equal_Hilbert} to the probability measure defined by the density $d(t)/(1-\lambda)$.\\
Thus $\mathcal{H}_\mu(x)$ is continuous on $\R\setminus\{x_0\}$,
 $\mathcal{H}_\mu(a)=\mathcal{H}_\mu(b)=\frac1{\pi C}$, and $\mathcal{H}_\mu(h(x))=\mathcal{H}_\mu(x)$ on $[a,b]$.\\

``$\Longleftarrow$'' Assume that $\mu$ satisfies (a), (b), and (c). 
We define a curve $\gamma:[a,b]\to\overline{\Ha}$ by 
\[\gamma(x) =\frac{1}{\pi(\mathcal{H}_\mu(x)-id(x))}=\frac{1}{\pi(\hat{\mathcal{H}}_\mu(x)-i d(x))}.\]
Then $\gamma$ is continuous with $\gamma(a)=\gamma(b)=C$ and $\gamma(a,u]=\gamma[u,b)\subset \Ha$.\\

Denote by $D$ the domain $D=F_\mu(\Ha)$. The points of $\partial D$ which are accessible from $D$,
denoted by $\partial_a D$, correspond to the limits  
$\lim_{\eps\downarrow 0}F_\mu(x+i\eps)$, see \cite[Exercises 2.5, 5]{MR1217706}.
Hence $\partial_a D = \R\cup \gamma[a,b]$ and $\overline{\partial_a D}=\partial_a D$.
As $\partial_a D$ is dense in $\partial D$, see \cite[Theorem 3.23]{Wil63}, we obtain 
$\partial D = \R\cup \gamma[a,b]$.
Hence, $F_\mu$ has a continuous extension to $\overline{\Ha}$, see \cite[Theorem 2.1]{MR1217706}.\\

Clearly, $D$ is the unbounded component of the complement of $\gamma[a,b]=\gamma[a,u]$ in $\Ha$. \\
Let $p\in\gamma(a,u)$. Due to the symmetry $h$ we know that $F_\mu^{-1}(\{p\})$ consists of at least $2$ 
points. Hence, by \cite[Proposition 2.5]{MR1217706}, 
$\gamma[a,u]\setminus \{p\}$ is not connected, i.e. $p$ is a cut-point of the curve $\gamma[a,u]$. 
We conclude that $\gamma[a,u]$ is a simple curve, e.g. by \cite[Theorem 1]{Ayr29}. (In this reference, 
$M$ should be taken as $\gamma[a,u]\cup J$, 
where $J$ is a simple curve in $\{C,\gamma(u)\}\cup\Ha\setminus\gamma[a,u]$ connecting 
$C$ and $\gamma(u)$.)
Hence $D=\Ha\setminus \gamma[a,u]$.


\end{proof}
\begin{remark}\label{unio}
The proof shows that $x_0<a$ if $C>0$ and $x_0>b$ if $C<0$. \\
Furthermore, we note that there is a unique $u\in (a,b)$ with $d(u)=u$. This number is equal to the preimage of
the tip of  $\gamma$ under the map $F_\mu$.\\

Assume that only the density $d$ on $[a,b]$ is known. Then $\lambda:=\mu(\{x_0\})$ can simply be determined by $\lambda=1-\int_a^b d(x)\, dx$. 
Furthermore,  
$\mathcal{H}_{\mu}(x)=\mathcal{H}_d(x)+\frac{\lambda}{\pi(x-x_0)}$. As $\frac1{\pi C}=\mathcal{H}_{\mu}(a)=\mathcal{H}_{\mu}(b)$, we see that $x_0$ satisfies the quadratic equation
$\frac{\lambda(b-a)}{\pi(\mathcal{H}_d(a)-\mathcal{H}_d(b))}=(x_0-a)(x_0-b)$.
\end{remark}

The case of a slit starting at $0$ is quite similar.

\begin{theorem}\label{slit_measures2} Let $\mu$ be a probability measure on $\R$ such that $F_\mu$
is univalent. \\
Then $F_\mu$ maps $\Ha$ conformally onto $\Ha\setminus \gamma$, 
where $\gamma$ is a slit starting at $C=0$, if and only if the following conditions 
are satisfied:
\begin{itemize}
 \item[(a)] $\supp \mu= [a,b]$, where $\mu$ 
 has a continuous density $d(x)>0$ on $(a,b)$. 
 \item[(b)] $\mathcal{H}_\mu$ is defined and continuous on $\R\setminus\{a,b\}$ with 
  $\lim_{x\downarrow a}|\mathcal{H}_\mu(x)|=\lim_{x\uparrow b}|\mathcal{H}_\mu(x)|=\infty$ or 
 $\lim_{x\downarrow a}d(x)=\lim_{x\uparrow b}d(x)=\infty$.
 \item[(c)] There exists a decreasing homeomorphism $h:[a,b]\to[a,b]$ with 
 \[d(h(x))=d(x) \qquad \text{and} \qquad \mathcal{H}_\mu(h(x))=\mathcal{H}_\mu(x)\]
 for all $x\in (a,b)$.
 \end{itemize}
\end{theorem}
\begin{proof}
``$\Longrightarrow$'':\\
We can argue as in the proof of $C\not=0$.
In this case, $F_\mu$ has the zeros $a,b$ and no zero in $\R\setminus[a,b]$.\\ 

 The Stieltjes-Perron inversion formula implies that $\supp \mu=[a,b]$ and that $\mu$ is 
 absolutely continuous on
 $(a,b)$ and the density  $d(x)$ as well as $\hat{\mathcal{H}}_\mu(x)$ are continuous on $(a,b)$ with 
 $d(h(x)) = d(x)$ and $\hat{\mathcal{H}}_\mu(h(x))=\hat{\mathcal{H}}_\mu(x)$ for all $x\in(a,b)$. 
 As the curve starts at $0$, its image under $z\mapsto -1/z$ 
 is a simple curve from some point in $\Ha$ to $\infty$ on the Riemann sphere. Hence 
 $|1/\gamma(x)|=|\pi(\hat{\mathcal{H}}_\mu(x)-i d(x))|\to \infty$ 
 as $x\downarrow a$ and  as $x\uparrow b$. Consequently, $d(x)\to\infty$ or $|\hat{\mathcal{H}}_\mu(x)|\to\infty$ as 
$x\downarrow a$ and as $x\uparrow b$.\\
 
It remains to show that $\hat{\mathcal{H}}_{\mu}$ and 
$\mathcal{H}_{\mu}$ coincide on $(a,b)$. Let $I$ be an open interval such that its closure is contained in 
$(a,b)$. We decompose the measure $\mu$ into two non-negative measures $\mu=\nu_1+\nu_2$, 
where $\nu_1(\overline{I})=0$, $\nu_2((-\infty,a+\eps)\cup(b-\eps,\infty))=0$ for some $\eps>0$. 
Furthermore, we require that $\nu_2$ has a continuous density.\\

We define $\hat{\mathcal{H}}_{\nu_j}, \mathcal{H}_{\nu_j}$ by integrating with respect to $\nu_j(dt)$. 
As $\nu_1(\overline{I})=0$, $\hat{\mathcal{H}}_{\nu_1}$ is continuous (in fact analytic) on $I$. 
Also $\mathcal{H}_{\nu_1}$ is defined on $I$ and it is easy to see that 
$\hat{\mathcal{H}}_{\nu_1}(x)=\mathcal{H}_{\nu_1}(x)$ on $I$.
We know that $\hat{\mathcal{H}}_\mu=\hat{\mathcal{H}}_{\nu_1}+\hat{\mathcal{H}}_{\nu_2}$ is continuous on $(a,b)$ and 
we conclude that $\hat{\mathcal{H}}_{\nu_2}$ exists and is continuous on $I$.\\
We now apply 
Lemma \ref{equal_Hilbert} to $\nu_2/\nu_2(\R)$ and obtain that 
$\mathcal{H}_{\nu_2}(x)$ exists and is equal to $\hat{\mathcal{H}}_{\nu_2}(x)$ on $I$. 
Thus $\mathcal{H}_\mu(x)=\mathcal{H}_{\nu_1}(x)+\mathcal{H}_{\nu_2}(x)=\hat{\mathcal{H}}_\mu(x)$ on $I$. As the 
interval $\overline{I}\subset (a,b)$ was chosen arbitrarily, this is true for 
the whole interval $(a,b)$.\\

``$\Longleftarrow$'':\\
Assume that $\mu$ is a probability measure on $\R$ satisfying (a), (b), (c). 
We define a curve $\gamma:(a,b)\to \Ha$ by 
$\gamma(x) =\frac{1}{\pi(\mathcal{H}_\mu(x)-i d(x))}=\frac{1}{\pi(\hat{\mathcal{H}}_\mu(x)-i d(x))}$.\\
Then $\gamma$ is continuous with $\gamma(a,u]=\gamma[u,b)\subset \Ha$ and 
$\lim_{x\downarrow a} \gamma(x) = \lim_{x\uparrow b}\gamma(x) = 0$.\\

The rest of the proof is analogous to the case $C\not=0$.
\end{proof}

\begin{remark}\label{surj}
 Assume that $\mu$ is a probability measure such that $F_\mu(\Ha)=\Ha\setminus \gamma$ 
 for a simple curve $\gamma$. Such an $F_\mu$ does not need to be injective:\\
  Let $G:\Ha \to \Ha\setminus \gamma$ be the unique conformal mapping 
 with $G(z)=z+\mathcal{O}(1/|z|)$ as $z\to\infty$ and let $H$ be the $F$-transform of 
$\frac1{2}\delta_{-1}+\frac1{2}\delta_1$. Then
\[H(z)=\frac1{\frac{1/2}{z+1}+\frac{1/2}{z-1}}=\frac{z^2-1}{z}=z-\frac1{z}\] and 
$H$ is surjective (rational function of degree $2$ mapping $\R\cup{\{\infty\}}$ onto itself) but not injective ($H(i/2\pm\sqrt{3}/2)=i$). Consequently, $G\circ H$ is a non-injective $F$-transform with $(G\circ H)(\Ha)=G(\Ha)=\Ha\setminus\gamma$.
\end{remark}

\begin{remark}\label{rm:1}
Note that $F_\mu(\Ha)=F_\mu(\Ha-d)=F_{\mu'}(\Ha)$ whenever $\mu'$ is $\mu$ translated by $d\in\R$. \\
Conversely, if we have two univalent $F$-transforms with 
$F_\mu(\Ha)=F_{\mu'}(\Ha)=\Ha\setminus \gamma$, then $\alpha=F_\mu\circ F_{\mu'}^{-1}$ is an automorphism 
of $\Ha$ with $\alpha(\infty)=\infty$ and $\alpha'(\infty)=1$, which implies $\alpha(z)=z+d$ for some 
$d\in\R$. Hence $\mu'$ is a translation of $\mu$.
\end{remark}

\begin{remark}If we know that 
$F_\mu$ is a univalent slit mapping, then, by the previous remark, the variance $\sigma^2$ of $\mu$ only
depends on the slit $\gamma$. If we translate the measure such that $F_\mu$ has hydrodynamic normalization, then the first moment 
of $\mu$ is equal to $0$ and we can see that the half-plane capacity $c$ of the slit is in fact equal to $\sigma^2$, see 
\cite[Proposition 2.2]{M92}.\\
The half-plane capacity has a more or less geometric interpretation, see
\cite{LLN09}. An explicit probabilistic formula is given in \cite[Proposition 3.41]{lawler05}.
\end{remark}

\begin{remark}The homeomorphism $h$ is also called the \emph{welding homeomorphism} of 
the slit $\gamma$.\\
A slit $\gamma$ is called \emph{quasislit} if $\gamma$ approaches $\R$ nontangentially and 
$\gamma$ is the image of a line segment under a quasiconformal mapping.
The theory of conformal welding implies: 
$\gamma$ is a quasislit if and only if $h$ is quasisymmetric;
see \cite[Lemma 6]{Lind:2005} and \cite[Lemma 2.2]{MarshallRohde:2005}.\\
In this case, the slit is uniquely determined by $h$ and its starting point $C$. An example of a slit which is not uniquely 
determined by $h$ and $C$ is a slit with positive area. \\
We refer to \cite{Bis07} for further results concerning conformal welding.
\end{remark}

\begin{example}
 Take a simple curve $\gamma:[0,1)\to\overline{\Ha}$ such that $\gamma(0)=0$, $\gamma(0,1)\subset \Ha\setminus[i,2i]$, 
 and the limit points of $\gamma$ as $t\to 1$ form the interval $[i,2i]$, as depicted in the figure 
 below. Let $D=\Ha\setminus (\gamma(0,1) \cup [i,2i])$. Then $D$ is simply connected. Let $F_\mu:\Ha\to D$ 
 be univalent. Then the limit $\lim_{\eps\downarrow 0}F_\mu(x+i\eps)$ exists for every 
 $x\in \R$ due to \cite[Exercises 2.5, 5]{MR1217706} and the fact that the prime end $p$ that corresponds to 
 $[i,2i]$ is accessible, i.e. the point $2i$ can be reached by a Jordan curve in $D$.
 In this case, $\mu$ has quite similar properties as in Theorem \ref{slit_measures2}, but the density $d$ is not continuous. The midpoint $u$ corresponds to the preimage of $p$ under $F_\mu$.\\
 If we replace the vertical interval $[i,2i]$ by a horizontal interval like $[i,1+i]$, a similar construction 
 yields a measure $\mu$ satisfying all properties as in Theorem \ref{slit_measures2} except that 
 $\mathcal{H}_\mu$ is not continuous. 
\end{example}

 \begin{figure}[h]
 \begin{center}
 \includegraphics[width=0.5\textwidth]{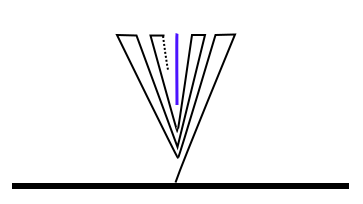}
 \caption{A curve $\gamma$ approaching a vertical line segment (blue).}
 \end{center}
 \end{figure}

\begin{example}
Consider the simply connected domain $D=\Ha\setminus \overline{\D}$. Let $F_\mu:\Ha\to D$ 
be univalent. The density $d$ of $\mu$ is symmetric 
with respect to the homeomorphism $h(x)=-x$, but $\mathcal{H}_\mu(h(x)) = -\mathcal{H}_\mu(x)$.
\end{example}

\subsection{Cauchy transforms vs Fourier transforms}\label{qu}

The Fourier transform of a probability measure $\mu$ is given as 
$\F_\mu(x)= \int_\R e^{ixt} \mu(dt)$, $x\in\R$. Classical independence of random variables 
leads to the classical convolution 
 $\mu\ast\nu$ defined by \[\F_{\mu\ast\nu}=\F_\mu \cdot \F_\nu.\]

 \begin{definition}
A stochastic process $(X_t)_{t\geq 0}$ is called an \emph{additive process} if the following three conditions are satisfied.
\begin{enumerate}
 \item[(1)] The increments $X_{t_0}, X_{t_1}-X_{t_0}, ..., X_{t_n}-X_{t_{n-1}}$ are independent 
 for any choice of $n\geq 1$ and $0\leq t_0 < t_1 <... < t_n$.
 \item[(2)] $X_0=0$ almost surely.
 \item[(3)] For any $\eps>0$ and $s\geq 0$, $\mathbb{P}[|X_{s+t}-X_s|>\eps]\to 0$ as $t\to 0$.
\end{enumerate}
Such a process is called a \emph{L\'evy process} if, in addition, 
 \begin{enumerate}
 \item[(4)] the distribution of $X_{t+s}-X_s$ does not depend on $s$.
\end{enumerate}
 \end{definition}
 
 \begin{definition}
A probability measure $\mu$ on $\R$ is said to be
\emph{$\ast$-infinitely divisible} if for every $n\in\N$ there exists
$\mu_n$ such that $\mu=\mu_n \ast ...\ast \mu_n$ ($n$-fold convolution). 
The set of all infinitely divisible distributions is denoted by
$\ID(\ast)$.
\end{definition}

The following result characterizes all distributions appearing in additive processes, see 
\cite[Theorems 1.1-1.3]{BNMR01}.

\begin{theorem}\label{inf_div}
 Let $\mu$ be a probability measure on $\R$. The following statements are equivalent:
 \begin{enumerate}
 \item[(a)] There exists an additive process $(X_t)_{t\geq0}$ 
 such that $\mu$ is the distribution of $X_1$.
 \item[(b)] There exists a L\'evy process $(X_t)_{t\geq0}$ 
 such that $\mu$ is the distribution of $X_1$.
 \item[(c)] $\mu\in \ID(\ast)$.
 \item[(d)] (L\'evy-Khintchine representation) 
 There exist $a\in \R$, $\sigma \geq 0$, and a non-negative measure $\nu$ with $\nu(\{0\})=0$ and 
 $\int_\R (1\wedge t^2) \nu(dt)<\infty$ such that
 \begin{equation} \label{khin}
\F_\mu(x) =\exp\left(iax - \frac1{2}\sigma^2x^2 + \int_{\R} \left(e^{ix t}-1- ixt 1_{\{|x|<1\}}
\right)\nu(dt)\right), \qquad x\in \R.
\end{equation}
\end{enumerate}
\end{theorem}

\begin{remark}
For $\mu\in\ID(\ast)$, we denote by $L(\mu)=(a,\sigma, \nu)$ the L\'evy triple of $\mu$. If we shift $\mu$ by a constant 
$c\in\R$, then we obtain $L(\mu(\cdot-c))=(a+c,\sigma,\nu)$.
\end{remark}

The $F$-transform plays the role of the Fourier transform in \emph{monotone probability theory}. The monotone 
convolution $\mu\rhd \nu$ is defined by \[F_{\mu\rhd\nu}=F_\mu\circ F_\nu.\]

The monotone analogue of property (a) in Theorem \ref{inf_div} is the property of $F_\mu$ being a univalent
function. 

\begin{theorem}[Theorem 1.16 in \cite{FHS}] Let $\mu$ be a probability measure on $\R$. The following statements are equivalent:
\begin{itemize}
	\item[(a)] $F_\mu$ is univalent.
	\item[(b)] There exists a quantum process $(X_t)_{t\geq0}$ with monotonically independent increments such that 
	$\mu$ is the distribution of $X_1$.
\end{itemize}
\end{theorem}

For the precise meaning of the quantum process mentioned in (b), we refer the reader to \cite{FHS}.

\begin{remark}
 Let $\mu_t$ be the distribution of $X_t$ and let $f_t=F_{\mu_t}$. In \cite[Proposition 3.11]{FHS} it is shown 
 that  $(f_t)$ is a decreasing Loewner chain, i.e. every $f_t$ is univalent and 
 $f_t(\Ha)\subset f_s(\Ha)$ 
 whenever $s\leq t$. In case $t\mapsto f_t$ is differentiable, it satisfies a Loewner equation 
 of the form \begin{equation*}
\frac{\partial}{\partial t} f_{t}(z) = \frac{\partial}{\partial z}f_{t}(z)\cdot M(z,t), 
\end{equation*}
where $M(z,t)=\gamma_t + \int_\mathbb{R}\frac{1+xz}{x-z} \rho_t({\rm d}x)$ 
for some $a_t\in\mathbb{R}$ and a finite non-negative measure $\rho_t$ on $\mathbb{R}$. \\

Let us compare this to classical additive processes:

 Let $(\mu_{t})_{t\geq0}$ be the distributions of a L\'evy process and let $\F_t=\F_{\mu_t}$. 
Then $(\F_t)_{t\geq0}$ is a multiplicative semigroup with $\F_1(x)=\F_\mu(x)=e^{\varphi(x)}$, 
i.e. $\F_t=e^{t\varphi(x)}$ or
\begin{equation*}
\frac{d}{dt}\F_t(x) = \varphi(x)\cdot \F_t, \qquad \F_0(x) \equiv 1.
\end{equation*}

The non-autonomous case of this equation is given by
\begin{equation*} 
\frac{d}{dt}\F_t(x) = \varphi_t(x)\cdot \F_t, \qquad \F_0(x) \equiv 1,
\end{equation*}
where $\exp(\varphi_t(x))=\F_{\nu_t}$ with $\nu_t\in\ID(\ast)$ for almost every $t\geq0$.
This equation corresponds to additive processes, provided that $t\mapsto\F_{\mu_t}(x)$ is indeed 
differentiable almost everywhere.
\end{remark}

By replacing the $F$-transform with the classical Fourier transform, we can ask some questions from 
Section 2 now for the Fourier transform.\\

Consider an additive process $(X_t)$ with distributions $(\mu_t)$. Let $(a_t, \sigma_t,\nu_t)=L(\mu_t)$. 
Then we can normalize the process by $Y_t=X_t - a_t$. Also $(Y_t)$ is an additive process and the distributions 
$(\alpha_t)$ of $(Y_t)$ satisfy $L(\alpha_t)=(0,\sigma_t,\nu_t)$.
 Let $\mu\in\ID(\ast)$. Let $(X_t)$ and $(Y_t)$ be two normalized additive processes such that 
 $\mu$ is the distribution of $X_1$ and of $Y_1$. \\
 We say that $\mu$ has a unique embedding if the distributions of $(Y_t)$ are obtained 
 by a time change of the distributions of $(X_t)$. 

\begin{theorem}Let $\mu\in \ID(\ast)$  with L\'evy triple $L(\mu)=(0,\sigma,\nu)$.
\begin{itemize}
 \item[(a)] $\mu$ can be embedded into an additive process $(X_t)$ with distributions $(\mu_t)$
 such that $t\mapsto\F_{\mu_t}$ is differentiable everywhere.
 \item[(b)]  
  $\mu$ has a unique embedding if and 
 only if $\nu=0$ or $\sigma=0$ and $\nu=\lambda \delta_{x_0}$ for some $x_0\in\R\setminus\{0\}$.
\end{itemize}
\end{theorem}
\begin{proof}${}$
 \begin{itemize}
 \item[(a)] Due to Theorem \ref{inf_div}, each $\mu\in \ID(\ast)$ can be embedded into a L\'evy process. 
 \item[(b)] Let $\mu$ be embedded into a normalized process $(X_t)$ with distributions 
 $\mu_t$. Let $L(\mu_t)=(0,\sigma_t,\nu_t)$. The L\'evy-It\^{o} decomposition yields two independent 
 additive processes 
 $(A_t)$ and $(B_t)$ with L\'evy triples $(0,\sigma_t,0)$ and $(0,0,\nu_t)$ respectively such that 
 $X_t=A_t+B_t$. We define the process $(Y_t)_{t\in[0,1]}$ by $Y_t=A_{2t}$ for $t\in [0,1/2]$ and 
 $Y_t=Y_{1/2}+B_{2(t-1/2)}$ for $t\in(1/2, 1]$. Then $Y_1=A_1+B_1$ has the distribution $\mu$.\\
 Let $\mu$ have a unique embedding. Then $(Y_t)$ is a time change of $(X_t)$ and this implies that 
 $\nu_t = 0$ for all $t\geq 0$, and thus $L(\mu)=(0,\sigma,0)$, or $\sigma_t=0$ for all $t\geq 0$ and 
 thus $L(\mu)=(0,0,\nu)$.\\
 Furthermore, suppose $L(\mu)=(0,0,\nu)$ with $\nu(\R)>0$ and $\nu$ is not of the form $\lambda\delta_{x_0}$. 
 Then the support of $\nu$ consists of at least two points and we can decompose $\nu$ into $\nu = \nu_1+\nu_2$ 
 for some positive measures $\nu_1, \nu_2$ having different supports. A similar construction shows that 
 the unique embedding of $\mu$ implies $\nu_1=0$ or $\nu_2=0$, a contradiction.\\ 
 
 Conversely, assume that $\mu$ with $L(\mu)=(0,\sigma,0)$ (or $L(\mu)=(0,0,\lambda\delta_{x_0})$) is embedded into 
 an additive process $(X_t)$ with distributions $\mu_t$ such that $L(\mu_s)=(0,\sigma_s,\nu_s)$ with 
 $\nu_s\not=0$ (or $\sigma_s\not=0$) for some $s<1$. Then $\mu = \mu_s \ast \mu_{s,1}$, where $\mu_{s,1}$ is the 
 distribution of $X_1-X_s$, 
 and this implies that $L(\mu)=(0,\sigma,\nu)$ for some $\nu\not=0$ (or $\sigma\not=0$), a contradiction. \\
Hence, $L(\mu_t)=(0,\sigma_t,0)$ (or $L(\mu_t)=(0,0,\lambda_t\delta_{x_0})$) for all $t\in[0,1]$ 
and $t\mapsto \sigma_t$ (or $t\mapsto \lambda_t$) is non-decreasing. Such processes are unique with respect to time changes. 
 

\end{itemize}
\end{proof}
\begin{remark}
 The cases from (b) correspond to the Dirac measure at $0$ ($\sigma=0, \nu=0$),
 the normal distribution with mean 0 and variance $\sigma^2$ ($\sigma>0, \nu=0$), and 
 the case $\sigma = 0, \nu=\lambda \delta_{x_0}\not=0$ corresponds to 
 certain Poisson distributions. Let $N$ be a Poisson random variable with parameter 
 $\lambda>0$ and let $X=x_0 \cdot N$ for some $x_0\not=0$. The distribution $\mu$ of $X$ satisfies 
 $\F_\mu(x) = \exp(\lambda(e^{ixx_0}-1))$. If $|x_0|\geq 1$, then $L(\mu)=(0,0,\lambda \delta_{x_0})$.\\
 If $|x_0|< 1$, then $L(\mu)=(\lambda x_0,0,\lambda \delta_{x_0})$. Hence, the distribution of 
 $X-\lambda x_0$ has the L\'evy triple $(0,0,\lambda\delta_{x_0})$.\\
 Our definition of a ``normalized additive process'' is somehow arbitrary, basically because the 
 cut--off function in representation \eqref{khin} can be replaced by others.\\
 One could also normalize 
 by subtracting the mean of $\mu$, provided it exists. 
 However, also there, we end up with the Dirac measure at $0$, the normal distribution with mean $0$ and 
 variance $\sigma^2$, and (scaled and shifted versions of) the Poisson distribution. 
 \end{remark}

\subsection{Problems}

\begin{problem}
 Let $\mu$ be a probability measure satisfying the conditions $(a)$-$(c)$ from 
Theorem \ref{slit_measures} or \ref{slit_measures2} respectively. Is $F_\mu$ necessarily univalent?
\end{problem}
\begin{problem}Which probability measures $\mu$ have a surjective $F$-transform, i.e. 
 $F_\mu(\Ha)=\Ha$.  An example of such $\mu$ is given in Remark \ref{surj}.
\end{problem}
\begin{problem}Motivated by Theorem  \ref{slit_measures}, 
 we can replace the Hilbert transform $\mathcal{H}_\mu$ by some other transform $T_\mu$ 
 and consider probability measures $\mu$ such that 
 there exists a decreasing homeomorphism $h:\R\to\R$ with
 \[\mu(A)=\mu(h(A)) \quad \text{and} \quad T_\mu\circ h=T_\mu\] 
for all Borel sets $A\subset\R$. \\
Is it true that in case $T_\mu=\F_\mu$ we necessarily have $h(x)=-x$?\\
An example is the normal distribution with mean $0$ and variance 
$\sigma^2$, where $\F_{\mu}(x)=e^{-\sigma^2x^2/2}$, and thus $h(x)=-x$.
\end{problem}

\subsection*{Acknowledgements}

The authors would like to thank Mihai Iancu for helpful discussions and the proof of Theorem 
\ref{S_d}.

%
%
%
%


\bibliographystyle{amsplain}

\end{document}